\documentclass[a4paper,12pt]{amsart}
\usepackage{amsmath,amssymb,amsthm,amsfonts,dsfont,latexsym,enumerate}
\usepackage[T1]{fontenc}
\usepackage[latin1]{inputenc}
\usepackage[english]{babel}
\usepackage{graphicx}
\graphicspath{{Figures_R_eps/}}
\usepackage{geometry}
\usepackage{natbib}
\usepackage{multirow}
\newtheorem{theorem}{Theorem}
\newtheorem{definition}{Definition}
\newtheorem{example}{Example}
\newtheorem{lemma}{Lemma}
\newtheorem{proposition}{Proposition}
\newtheorem{remark}{Remark}
\renewcommand{\qedsymbol}{}
\def\qedsymbol{$\blacksquare$}
\newcommand{\indicatrice}{\mathds{1}}

\begin{document}
\title[Bivariate Cox model and copulas]{Bivariate Cox model and copulas}
\author{Mohamed Achibi}
\address[Mohamed Achibi]
{Université Pierre et Marie Curie\\
LSTA-Paris 6\\
175 Rue du Chevaleret\\
75013 Paris\\
France
\newline
Snecma\\
Site de Villaroche\\
Rond-Point René Ravaud\\
77550 Moissy Cramayel\\
France}%
\email{mohamed.achibi@snecma.fr}%
\author{Michel Broniatowski}
\address[Michel Broniatowski]
{Université Pierre et Marie Curie\\
LSTA-Paris 6\\
175 Rue du Chevaleret\\
75013 Paris\\
France\\
}%
\email{michel.broniatowski@upmc.fr}%
\date{June 18, 2010}
\begin{abstract}
This paper introduces a new class of Cox models for dependent bivariate data. The impact of the covariate on the dependence of the variables is captured through the modification of their copula. Various classes of well known copulas are stable under the model (archimedean type and extreme value copulas), meaning that the role of the covariate acts in a simple and explicit way on the copula in the class; specific parametric classes are considered.
\end{abstract}
\maketitle
\section{Introduction} \label{intro} 
The aim of this paper is to present a new description for bivariate dependence. It extends the proportional hazard (PH) model and is relevant in various fields of biostatisctics and industry. Denote $z$ an
environmental covariate and consider two positive random variables $X$ and $Y$ with absolutely continuous survival functions 
\begin{equation}
\overline{F}^{{z}}(x)=\mathbb{P}(X>x;z)  \label{survie_F}
\end{equation}
and 
\begin{equation}
\overline{G}^{{z}}(y)=\mathbb{P}(Y>y;z)  \label{survie_G}
\end{equation}
under the covariate $z$. The joint distribution of $(X,Y)$ is modelled through (\ref{survie_F}) and (\ref{survie_G}) and through the conditional s.d.f
\begin{equation}
\overline{G}_{x}^{{z}}(y)=\mathbb{P}(Y>y|X>x;z).  \label{survie_Gz}
\end{equation}
The precise setting of our model is as follows: we assume that the survival function of $X$ depends on $z$ through a PH model, i.e 
\begin{equation}
\lambda _{X}^{z}(x):=-\frac{d}{dx}\log \overline{F}^{z}(x)=\lambda
_{X}^{0}(x)\Phi (z)  \label{Cox_X}
\end{equation}
for some positive function $z\rightarrow \Phi (z)$ and some baseline hazard $\lambda _{X}^{0}(x)$ corresponding to $z=0.$
Denote 
\begin{equation*}
\lambda _{\left. Y\right\vert X>x}^{z}(y):=-\frac{d}{dy}\log \overline{G}%
_{x}^{{z}}(y).
\end{equation*}
The conditional survival function of $Y$ depends on $z$ through the PH model
\begin{equation}
\lambda _{\left. Y\right\vert X>x}^{z}(y)=\lambda _{\left. Y\right\vert
X>x}^{0}(y)\Psi \left( z\right)   \label{Cox_YX}
\end{equation}
for some positive function $z\rightarrow \Psi (z)$ and some baseline hazard $%
\lambda _{\left. Y\right\vert X>x}^{0}.$
Denote 
\begin{equation}
{(\text{M}1)}\left\{ 
\begin{array}{c}
\lambda _{\scriptscriptstyle{X}}^{z}(x)=\lambda _{\scriptscriptstyle{X}}^{%
\scriptscriptstyle{0}}(x)\Phi (z) \\ 
\lambda _{\scriptscriptstyle{Y|X>x}}^{z}(y)=\lambda _{\scriptscriptstyle{%
Y|X>x}}^{\scriptscriptstyle{0}}(y)\Psi (z)%
\end{array}%
\right.   \label{Cox cond X>x}
\end{equation}
Other models for bivariate dependence have been defined, for example \citep{Clayton1978, Clayton_Cuzick, Oakes1989}. In \citet{Clayton1978} where joint lifetimes of sons and fathers are considered, a Cox model is used to handle the role of the covariate upon the margins, and the association between the margins is described through the fact that the local association measure
\begin{equation*}
\theta \left( x,y\right) :=\frac{\lambda \left( \left. y\right\vert
X=x\right) }{\lambda \left( \left. y\right\vert X>x\right) }
\end{equation*}
is independent upon $x$ and $y$ and upon $z.$ This index expresses the influence of parental history of a given disease upon the incidence in the offspring. $X$ is the lifetime of the father, while $Y$ describes that of
the son. It is assumed that association arises because the two members of a given pair share some common influence and not because one lifetime influences the other. In fact these models are of frailty type, and the
margins are independent conditionally upon the covariate. More globally the standard model building strategy is to have marginal survival functions and a copula for their dependency; see f.i. \citet{Choi_Matthews} or \citet{Said}.
In \citet{Lindstrom} a study on the familial concordance in cancer survival based on a Swedish population showed that cancer specific survival in parents predicts survival for the same cancer in their children. The risk in dying in children in relation to parental survival was modelled by use of two PH models; first parental survival was modelled and next survival risk in children in relation to parental survival was assessed; thus in this case the lifetime of the father influences the lifetime of the son whatever the level of the covariate (here  $z$ represents the different cancer sites), however in relation with it. In this case Model (\ref{Cox cond X>x}) seems to be natural.
Setting $x=0$ in  (\ref{Cox_YX}) shows that $Y$ follows a PH model so that (\ref{Cox cond X>x}) is PH on both components. The copula of the couple of r.v's $\left( X,Y\right) $ for a given $z$ can be written in terms of the values of $\Phi (z)$ and $\Psi (z),$ as seen further. The unusual feature of this model is positive in a number of cases, since $z$ acts on the margins, but also specifically on their dependence. The claim that the goal of copula modelling is to distinguish the parameters for the dependency from those associated to the marginal models cannot be considered as a general principle. Other authors have considered cases when the association of the margins is specifically related to their distributions; see f.i. \citet{Gupta2008}. In conjunction with (\ref{Cox cond X>x}) it is worth noting that direct approaches based on regression type models cannot satisfy our purpose. Indeed consider for example a model defined through
\[
\left\{
\begin{array}
[c]{c}%
X=r(z,U)\\
Y=s(z,V)
\end{array}
\right.
\]
with $r(z,.)$ and $s(z,.)$ strictly increasing for all $z$. Then following \citet{Nelsen2006}, Theorem 2.4.3, $(X,Y)$ has the same copula as $(U,V)$ for all $z$, which implies that the covariate $z$ plays no role in the dependency of $X$ and $Y$.
Introducing a new model imposes to determine its range of applicability; it will be shown that (\ref{Cox cond X>x}) is adapted for positive quadrant dependence (PQD) between the margins, a concept which is recalled in Section \ref{rappel}. Stating that the margins follow a PH (typically Cox) model can be checked using standard tools \citep[see f.i.][]{Grambsch}. PQD property can be tested through a Kolmogorov-Smirnov test \citep[see f.i.][]{Scaillet}.
\bigskip Let us now show the main results which we present in connection with model (\ref{Cox cond X>x}):
\begin{enumerate}
\item The $TP_2$ class of sdf's is a subclass of the Positive Quadrant Dependence (PQD) sdf class and is stable under the model which is properly defined when the hazard baseline $\overline{H}^{\scriptscriptstyle{0}}$ is $%
TP_2$. This class appears quite naturally as the one under which the model is properly defined, and it is appropriate for the description of positive dependence between its margins. Definitions of $TP_2$ and PQD properties are given in Section \ref{rappel}.
\item Since the $TP_2$ property of a multivariate sdf refers only to its copula, model (\ref{Cox cond X>x}) describes the changes of the baseline copula induced by the covariate. Also this implies that the model is valid
independently of marginal distributions. Only the structure of dependence is involved in the domain of validity of the model.
\item Two main classes of copulas are stable under the model namely: when the baseline bivariate copula is in such a class, so is the copula for all value of the covariate $z$. The class of extreme values copulas (evc) enjoys this property. The class of extended archimedean copulas is also stable under the model. This class results as a special by-product of a technique intended to produce asymmetric copulas due to Genest et al.; see %
\citet{Discussion_Genest} and \citet{Liebscher}. The so-called class of logistic asymmetric copulas \citep[see][]{Tawn1988}, which is a simple extension of the Gumbel family of copulas, enjoys an important role in the present model. It is stable under the model and admits a simple parametrization. The covariate $z$ acts in an adaptive way when the value of the covariate is changed. It is the only bivariate distribution in the class
of frailty models which enjoys such properties in the model.
\end{enumerate}
This paper is organized as follows. In Section \ref{rappel} we briefly recall the necessary background from bivariate dependence. Section \ref{covariate} describes the model. In Section \ref{logistic} we focus on the
asymmetric Gumbel class of copulas, which is the natural parametric setting of our model; we also provide some connection with bivariate frailty models. All proofs are deferred to the Appendix.
\section{Some useful facts in bivariate dependence} \label{rappel}
Let $X$ and $Y$ be two random variables (r.v) with joint sdf $\overline{H}$, with margins $\overline{F}$ and $\overline{G}$. All dependence properties of $X$ and $Y$ are captured through the \textit{survival copula} $C$ which is a cdf defined on $[0,1]\times [0,1]$ through
\begin{eqnarray*}
C(u,v) &=&\overline{H}(\overline{F}^{\leftarrow}(u),\overline{G}^{\leftarrow}(v))
\end{eqnarray*}
where $u$ and $v$ belong to $[0,1]$ and where $\overline{F}^{\leftarrow}(t):=\sup \{x:\overline{F}(x)\geq t\}$. It is easily checked that $C$ is indeed a copula. The definition of a copula is given in \citet{Nelsen2006}, definition 2.2.2.
We will make use of the following definition and notation.
\begin{definition}[Min-id property]
A bivariate cdf $H$ is min-infinitely divisible (min-id) if for all positive $\gamma$, $\overline{H}^{\gamma}$ is a sdf.
\end{definition}
Assume that $H$ is min-id and let $\mathbb{V}=(X,Y)$ be a random vector with sdf $\overline{H}$. Then for all $n$ in $\mathbb{N}$, $\overline{H}^{1/n}$ is a sdf. Further let $\left(X_{i},Y_{i}\right), \,i=1,\ldots ,n$ be $n$ copies independent and identically distributed with sdf $\overline{H}^{1/n}.$ It holds 
\begin{equation*}
\mathbb{V}\overset{d}{=}(\min_{i}~X_{i},\min_{i}~Y_{i}).
\end{equation*}
\begin{definition}[PQD property]
$X$ and $Y$ are positively quadrant dependent (PQD) iff, for all $(x,y)$ in $\mathbb{R}^{2},\mathbb{P}(X>x,Y>y)\geq \mathbb{P}(X>x)\mathbb{P}(Y>y)$; in this case we also say that $\overline{H}$ is PQD.
\end{definition}
\begin{definition}[$TP_2$ property]
A mapping $\phi$ from $\overline{\mathbb{R}}^{2}$ onto $\mathbb{R}$ is totally positive of order 2 ($TP_{2}$) if $\phi(x,y)\geq 0$ for all $(x,y)$ in $\overline{\mathbb{R}}^{2}$ and 
$\begin{vmatrix}
\phi (x_{1},y_{1}) & \phi (x_{1},y_{2}) \\ 
\phi (x_{2},y_{1}) & \phi (x_{2},y_{2})%
\end{vmatrix}%
=\phi (x_{1},y_{1})\phi (x_{2},y_{2})-\phi (x_{1},y_{2})\phi
(x_{2},y_{1})\geq 0$, for all $x_{1}<x_{2}$ and $y_{1}<y_{2}$.
\end{definition}
\begin{remark} \label{rq 1}
When $\phi$ is $\mathcal{C}^{2}$, then $\phi $ is $TP_{2}$ iff 
\begin{equation}  \label{TP2_derivee}
\frac{\partial \phi }{\partial x}(x,y)\frac{\partial \phi }{\partial y}(x,y)\leq \frac{\partial ^{2}\phi }{\partial x\partial y}(x,y)\phi (x,y)
\end{equation}
\end{remark}
The proof is given in \citet{Resnick1987}, p.254.\\
\newline We also recall the following results.
\begin{theorem}[\citet{Joe1997}, Theorem 2.3] \label{relation_dep}
If $\overline{H}$ is a $TP_{2}$ sdf, then $\overline{H}$ is $PQD$.
\end{theorem}
\begin{theorem}[\citet{Joe1997}, Theorem 2.6] \label{min-id TP2}
Let $H$ be a cdf, then $H$ is min-id iff $\overline{H}$ is $TP_{2}$.
\end{theorem}
The relation between properties of the s.d.f's and their copulas is captured in the following result.
\begin{lemma} \label{min inf div}
Let $\overline{H}$ be a sdf with copula $C$. Then $\overline{H}$ is $TP_2$ iff $C$ is $TP_2$.
\end{lemma}
\begin{proof}
By Remark \ref{rq 1}, $\overline{H}\; TP_2 \Rightarrow \frac{\partial \overline{H}}{\partial x}(x,y)\frac{\partial \overline{H}}{\partial y}(x,y)\leq \frac{\partial^2 \overline{H}}{\partial x \partial y}(x,y) \times \overline{H}(x,y)$. Furthermore, $\overline{H}(x,y)=C(\overline{F}(x),\overline{G}(y))$. Therefore,
\begin{eqnarray*}
\frac{\partial \overline{H}}{\partial x}(x,y)&=&-\frac{\partial C}{\partial u}(u,\overline{G}(y))\Big|_{u=\overline{F}(x)} \times f(x) \\
\frac{\partial \overline{H}}{\partial y}(x,y)&=&-\frac{\partial C}{\partial v}(\overline{F}(x),v)\Big|_{v=\overline{G}(y)} \times g(y) \\
\frac{\partial^2 \overline{H}}{\partial x \partial y}(x,y)&=&\frac{\partial^2 C}{\partial u \partial v}(u,v)\Big|_{\substack{ u=\overline{F}(x) \\ v=\overline{G}(y)}} \times f(x) \times g(y)
\end{eqnarray*}
Hence,
\begin{equation*}
\frac{\partial C}{\partial u}(u,\overline{G}(y))\Big|_{u=\overline{F}(x)}\times \frac{\partial C}{\partial v}(\overline{F}(x),v)\Big|_{v=\overline{G}(y)}\leq \frac{\partial^2 C}{\partial u \partial v}(u,v)\Big|_{\substack{ u=\overline{F}(x)  \\ v=\overline{G}(y)}}\times C(\overline{F}(x),\overline{G}(y))
\end{equation*}
\end{proof}
\begin{definition}[Archimedean copula]
An \textit{Archimedean copula} is a function $C$ from $[0,1]^2$ to $[0,1]$ given by $C(u,v)=\varphi^{[-1]}(\varphi(u)+\varphi(v))$, where $\varphi$ is a continuous strictly decreasing convex function from $[0,1]$ to $[0,\infty]$ such that $\varphi(1)=0$, and where $\varphi^{[-1]}$ denotes the "pseudo-inverse" of $\varphi$, namely
\begin{equation*} {\varphi^{[-1]}(t)=}
\left\{ 
\begin{array}{cc}
\varphi^{-1}(t) & \text{for t in}\; [0,\varphi(0)] \\ 
0 & \text{for}\; t\geq \varphi(0)%
\end{array}%
\right.
\end{equation*}
When $\varphi(0)=\infty$, $\varphi$ is said to be strict and $\varphi^{[-1]}\equiv \varphi^{-1}$. $\varphi$ is called a generator.
\end{definition}
The class of so called \textit{extreme value copulas} (evc) is important in this model, although not related here with the theory of bivariate extremes; therefore we define an \textit{extreme value copula} through the basic Pickands representation, without further reference to the theory of bivariate extremes.
\begin{theorem}[Pickands Theorem] \label{Pickands theo}
$C$ is an \textit{extreme value copula iff there exists a convex function} $A$ defined on $[0,1]$, which satisfies $A(0)=A(1)=1$ and $\max (t,1-t)\leq A(t)\leq 1$ such that
\begin{equation} \label{Pickands}
C(u,v)=\exp {\Big[\log (uv)A\Big(\frac{\log v}{\log uv}\Big)\Big]}.
\end{equation}
\end{theorem}
The function $A$ is referred to as the \textit{dependence function} or \textit{Pickands} function of the copula $C$.
\section{Introducing covariates in dependence models} \label{covariate}
\subsection{Description of the model}
Not all baseline survival d.f's $\overline{H}^{\scriptscriptstyle{0}}$ defines a model, so that $\lambda{_{\scriptscriptstyle{X}}^z}$ and $\lambda{_{\scriptscriptstyle{Y\left\vert X>x\right.}}^z}$ are the marginal and conditional specific cause hazards for some bivariate sdf $\overline{H}^{z}$ with margins $\overline{F}^{z}$ and $\overline{G}^{z}$ under a given covariate $z$. We conclude from the first equation of \eqref{Cox cond X>x} that $\overline{F}^{z}(x)=\left( {\overline{F}^{\scriptscriptstyle{0}}}(x)\right)^{\Phi (z)}$. By the second equation in \eqref{Cox cond X>x}, plugging $x=0$, we get $\overline{G}^{z}(y)=\left( \overline{G}^{\scriptscriptstyle{0}}(y)\right) ^{\Psi (z)}$. The model is defined when $z$ holds if $\left( {\overline{F}^{\scriptscriptstyle{0}}}(x)\right) ^{\Phi(z)}\left(\overline{H}{_{\scriptscriptstyle{Y|X>x}}^{\scriptscriptstyle{0}}}(y)\right)^{\Psi(z)}$ defines a sdf. Notice that
\begin{equation} \label{H bar z sdf}
\overline{H}^{z}(x,y)=\left(\overline{H}^{\scriptscriptstyle{0}}(x,y)\right)^{\Psi(z)}\left({\overline{F}^{\scriptscriptstyle{0}}}(x)\right)^{\Phi(z)-\Psi(z)}
\end{equation}
which is indeed a sdf when $\Phi(z)\geq \Psi(z)>0$ and $\left(\overline{H}^{\scriptscriptstyle{\scriptscriptstyle{0}}}(x,y)\right)^{\Psi (z)}$ is a sdf.
\newline
Also not all bivariate survival d.f's $\overline{H}^{\scriptscriptstyle{0}}$ are such that for all positive $\gamma$, $\left( \overline{H}^{\scriptscriptstyle{0}}\right) ^{\gamma}$ is a sdf. Min-infinite divisibility of the baseline hazard seems to be a natural assumption here. Assume therefore that:
\begin{equation}
H^{\scriptscriptstyle{0}} \; \text{is min-infinitely divisible} \tag{\textbf{H}} \label{H}
\end{equation}
By Theorem \ref{min-id TP2} and Lemma \ref{min inf div}, (\textbf{H}) holds iff $C_{\overline{H}^{\scriptscriptstyle{0}}}$ is $TP_2$. We have the following result.
\begin{proposition} \label{Prop Hbar sdf}
When (\ref{H}) holds then $\overline{H}^{z}$ defined in (\ref{H bar z sdf}) is a sdf for all $z$ such that $\Phi (z)\geq \Psi (z)>0$.
\end{proposition}
Let us consider the case when $0< \Phi(z)\leq \Psi(z)$. Analogously with (\ref{Cox cond X>x}) the model may then be written
\begin{equation} {(\text{M}2)} \label{Cox Phi less Psi}
\left\{ 
\begin{array}{c}
\lambda_{\scriptscriptstyle{Y}}^{z}(y)=\lambda_{\scriptscriptstyle{Y}}^{\scriptscriptstyle{0}}(y)\Psi(z) \\ 
\lambda_{\scriptscriptstyle{X|Y>y}}^{z}(x)=\lambda_{\scriptscriptstyle{X|Y>y}}^{\scriptscriptstyle{0}}(x)\Phi(z)%
\end{array}%
\right.
\end{equation}
permuting the role of $X$ and $Y$. In a similar way to the above we have that 
\begin{equation*}
\overline{H}^{z}(x,y)=\left( \overline{H}^{\scriptscriptstyle{0}}(x,y)\right)^{\Phi (z)}\left(\overline{G}^{\scriptscriptstyle{0}}(y)\right)^{\Psi (z)-\Phi (z)}
\end{equation*}
is a proper sdf. To summarize the above arguments we state:\\
\newline Let the model be defined by \eqref{Cox cond X>x} if $\Phi (z)\geq \Psi (z)$ and by \eqref{Cox Phi less Psi} if $\Phi (z) < \Psi (z)$. Call (\textbf{M}) the model defined through

\begin{equation*}
(\textbf{M}):=(M1)\indicatrice_{\Phi (z)\geq \Psi (z)} +(M2)\indicatrice_{\Phi (z)<\Psi (z)}.
\end{equation*}
\newline This model is well defined, even if $\Phi(z)$ and $\Psi(z)$ are not ordered uniformly on the covariate $z$ (which can be multivariate); the functions $\Phi$ and $\Psi$ can be easily estimated through the data, since they characterize the marginal PH models in (\textbf{M}). Suppose that $X$ and $Y$ are fitted to the same scale under the baseline, namely $\overline{F}^{\scriptscriptstyle{0}}(t)=\overline{G}^{\scriptscriptstyle{0}}(t)$ for all $t$. Then $\Phi (z)\geq \Psi (z)$ implies $\overline{F}^{z}(t)\leq \overline{G}^{z}(t)$ for all $t$, stretching the fact that $X$ becomes stochastically smaller than $Y$ under the stress parameter $z$.\\
\newline Identifiability of (\textbf{M}) holds; assume for example that $\Phi(z)\geq \Psi (z)$. Then $\Phi(z)$ and $\Psi(z)$ are defined uniquely. Indeed, assume 
\begin{eqnarray*}
\overline{H}^{z}(x,y)&=&\left(\overline{H}^{\scriptscriptstyle{0}}(x,y)\right)^{\Psi(z)}\left(\overline{F}^{\scriptscriptstyle{0}}(x)\right)^{\Phi (z)-\Psi (z)} \\
&=& \left( \overline{H}^{\scriptscriptstyle{0}}(x,y)\right)^{\Psi^{\prime }(z)}\left(\overline{F}^{\scriptscriptstyle{0}}(x)\right)^{\Phi^{\prime }(z)-\Psi^{\prime }(z)},
\end{eqnarray*}
for all $x,y$. Then taking logarithms yields $\Phi (z)=\Phi ^{\prime }(z)$ and $\Psi(z)=\Psi ^{\prime }(z)$.\\
When (\ref{H}) holds then for all $z$ , $\overline{H}^{z}$ is a sdf and
\begin{eqnarray}
\overline{H}^{z}(x,y) &=&\indicatrice_{\Phi (z)\geq \Psi (z)} \left(\overline{H}^{\scriptscriptstyle{0}}(x,y)\right) ^{\Psi (z)}\left( \overline{F}^{\scriptscriptstyle{0}}(x)\right) ^{\Phi (z)-\Psi (z)}  \label{Hz w.r.t. F0} \\
&&+\indicatrice_{\Phi (z) < \Psi (z)} \left( \overline{H}^{\scriptscriptstyle{0}}(x,y)\right)^{\Phi (z)}\left( \overline{G}^{\scriptscriptstyle{0}}(y)\right) ^{\Psi(z)-\Phi(z)}.  \notag
\end{eqnarray}
Min-infinite divisibility of the baseline will also make any $H^{z}$ min-infinitely divisible, showing that this class is \textit{stable} under (\textbf{M}). Indeed for any positive $\gamma$,
\begin{eqnarray*}
\left(\overline{H}^{z}(x,y)\right)^{\gamma} &=&\indicatrice_{\Phi(z)\geq \Psi(z)} \left(\overline{H}^{\scriptscriptstyle{0}}(x,y)\right)^{\gamma\Psi(z)}\left(\overline{F}^{\scriptscriptstyle{0}}(x)\right)^{\gamma[\Phi(z)-\Psi(z)]} \\
&&+\indicatrice_{\Phi (z) < \Psi(z)} \left( \overline{H}^{\scriptscriptstyle{0}}(x,y)\right) ^{\gamma\Phi (z)}\left(\overline{G}^{\scriptscriptstyle{0}}(y)\right)^{\gamma[\Psi(z)-\Phi(z)]},
\end{eqnarray*}
which still is a sdf.
By Theorem \ref{min-id TP2} and Lemma \ref{min inf div} min-infinite divisibility is not a property of the cdf but of its copula. Formula (\ref{Hz w.r.t. F0}) can be written for copulas through
\begin{eqnarray}
C_{\overline{H}^{z}}(u,v) &=&\indicatrice_{\Phi (z)\geq \Psi (z)} u^{\frac{\Phi (z)-\Psi (z)}{\Phi (z)}}C_{\overline{H}^{\scriptscriptstyle{0}}}\left( u^{\frac{1}{\Phi (z)}},v^{\frac{1}{\Psi (z)}}\right) ^{\Psi (z)}
\label{CopHz w.r.t. CopF0} \\
&&+\indicatrice_{\Phi (z) < \Psi (z)} v^{\frac{\Psi (z)-\Phi (z)}{\Psi(z)}}C_{\overline{H}^{\scriptscriptstyle{0}}}\left( u^{\frac{1}{\Phi (z)}},v^{\frac{1}{\Psi (z)}}\right)^{\Phi (z)}.  \notag
\end{eqnarray}
(\ref{H}) is only a sufficient condition for the model to be defined. The following example illustrates this fact.
\begin{example}
Let $C_{\overline{H}^{\scriptscriptstyle{0}}}(u,v)=uv\exp \left\{-\theta \log u\log v\right\}$, with $\theta \in (0;1]$. This is the \textit{Gumbel-Barnett} family \citep[see][p.119]{Nelsen2006}. By (\ref{TP2_derivee}) it is easy to check that $C_{\overline{H}^{\scriptscriptstyle{0}}}$ is not $TP_2$. Using (\ref{CopHz w.r.t. CopF0}) and assuming $\Phi (z)\geq \Psi (z)$ we obtain 
\begin{equation*}
C_{\overline{H}^{z}}(u,v)=uv\exp \left\{ -\frac{\theta }{\Phi (z)}\log u\log v\right\}
\end{equation*}
which still is a Gumbel-Barnett copula when $\frac{\theta }{\Phi (z)}$ belongs to $(0;1]$.
\end{example}
This example shows that (\ref{H}) is indeed the only acceptable condition for existence. Otherwise the baseline hazard $\overline{H}^{\scriptscriptstyle{0}}$ defines a model only for specific values of the covariate. This motivates our interest in good classes of min-infinitely divisible copulas which we intend to regress on the covariate $z$.
\subsection{Stability properties of the model} \label{stability}
We introduce two classes of copulas which are stable under (\textbf{M}).
\subsubsection{Extended archimedean copulas} \label{archimedean}
Among all possible types of bivariate dependence which can be described through the present bivariate Cox model, there exists a class of copulas which contains the archimedean copulas and which enjoys peculiar stability properties.\\
\newline
From now on, denote
\begin{equation}\label{not1}
\alpha(z)=\min\left(\frac{\Psi(z)}{\Phi(z)},1\right)
\end{equation}
and
\begin{equation}\label{not2}
\beta(z)=\min\left(\frac{\Phi(z)}{\Psi(z)},1\right).
\end{equation}
Let $\varphi_{\scriptscriptstyle{0}}$ be a generator and $C_{\varphi_{\scriptscriptstyle{0}}}$ be the archimedean copula with generator $\varphi_{\scriptscriptstyle{0}}$. We assume that $C_{\varphi_{\scriptscriptstyle{0}}}$ is $TP_2$. When $\overline{H}^z$ is defined, set $C_{\overline{H}^z}$ its copula.
\begin{proposition} \label{Proposition archimedean copula}
Let $\overline{H}^{\scriptscriptstyle{0}}$ be a sdf with $TP_2$ copula $C_{\varphi_{\scriptscriptstyle{0}}}$. Then, $\overline{H}^z$ is defined for all $z$ in the domain of $\Phi$ and $\Psi$. Further 
\begin{equation}\label{prop archimedean}
C_{\overline{H}^{z}}(u,v)= u^{1-\alpha(z)}v^{1-\beta(z)}\varphi_z^{-1}\left(\varphi_z\left(u^{\alpha(z)}\right)+\varphi_z\left(v^{\beta(z)}\right)\right)
\end{equation}
with
\begin{equation} \label{phiz}
\varphi_z(t)=\varphi_{\scriptscriptstyle{0}}\left(t^{\frac{1}{\Phi(z)\alpha(z)}}\right) = \varphi_{\scriptscriptstyle{0}}\left(t^{\frac{1}{\Psi(z)\beta(z)}}\right).
\end{equation}
\end{proposition}
More generally we have, denoting $\Pi(u,v)=uv$
\begin{proposition} \label{Generalisation archimedean copula}
Assume that 
\begin{equation}
C_{\overline{H}^{\scriptscriptstyle{0}}}(u,v) = \Pi(u^{1-\kappa},v^{1-\eta})C_{\varphi_{\scriptscriptstyle{0}}}(u^{\kappa},v^{\eta}), \, 0\leq \kappa, \eta \leq 1
\end{equation}
then
\begin{equation} \label{asymmetric copula}
C_{\overline{H}^z}(u,v) = \Pi(u^{1-\alpha(z)\kappa},v^{1-\beta(z)\eta})C_{\varphi_z}(u^{\alpha(z)\kappa},v^{\beta(z)\eta})
\end{equation}
where $C_{\varphi_z}$ denotes the archimedean copula with generator $\varphi_z$ defined in (\ref{phiz}).\\
\end{proposition}
\subsubsection{Extreme values copulas} \label{evc}
We show that the class of evc's also enjoys stability properties, as seen in the present Section. Extreme values copulas are $TP_2$ \citep{Hurlimann2003}.
\begin{proposition} \label{Proposition B fonction de A} 
When $\overline{H}^{\scriptscriptstyle{0}}$ has an evc $C_{\overline{H}^{\scriptscriptstyle{0}}}$ with Pickands function $A$ then, denoting $C_{\overline{H}^{z}}$ the copula of $\overline{H}^{z}$ and using \eqref{CopHz w.r.t. CopF0}, we have 
\begin{equation} \label{evc_z}
C_{\overline{H}^{z}}(u,v)=\exp \left[ \log (uv)B^{z}\left( \frac{\log v}{\log uv}\right) \right]
\end{equation}
with 
\begin{equation} \label{B dependence}
B^{z}(s)=1-W(z)K(z)-sW(z)[1-K(z)]+W(z)[(1-s)K(z)+s]A\left( \frac{s}{K(z)(1-s)+s}\right)
\end{equation}
where $K(z)=\frac{\Psi (z)}{\Phi (z)}\;\text{and}\;W(z)=\min \left( \frac{1}{K(z)},1\right)=\beta(z)$.
\end{proposition}
\begin{remark}
This basic result shows that $C_{\overline{H}^{z}}$ is an evc with Pickands function $B^{z}$. Proposition \ref{Proposition B fonction de A} shows that the class of evc's\ is \textit{stable} under (\textbf{M}). Although the copula of $\overline{H}^{\scriptscriptstyle{0}}$ is an evc, this does not imply in any respect that its marginals should be extreme value sdf's.
\end{remark}
From (\ref{B dependence}) we deduce the \textit{transition formula} which links $B^{z^{\prime }}$ to $B^{z}$ for two different values of the covariate. It holds
\begin{proposition} \label{Proposition transition B}
Under (\textbf{M}) let $\overline{H}^{\scriptscriptstyle{0}}$ has an evc. Then with the above notation, for all $z$, $z^\prime$,
\begin{equation} \label{transition Bz'_Bz}
B^{z^{\prime }}(s)=1 - \frac{\alpha (z^{\prime })}{\alpha (z)} + \left(\frac{\alpha (z^{\prime })}{\alpha (z)} - \frac{\beta (z^{\prime })}{\beta (z)}\right)s + \left[ \left( 1-s\right) \frac{\alpha (z^{\prime })}{\alpha (z)}+s\frac{\beta (z^{\prime })}{\beta (z)}\right] B^{z}\left( \frac{s\frac{\beta (z^{\prime })}{\beta (z)}}{\left(1-s\right) \frac{\alpha (z^{\prime })}{\alpha(z)}+s\frac{\beta (z^{\prime })}{\beta (z)}}\right)
\end{equation}
\end{proposition}
Proposition \ref{Proposition transition B} proves that the transition from $z$ to $z^{\prime}$ is independent of the baseline dependence function. Formula \eqref{transition Bz'_Bz} can be seen as a kind of expression of the proportional hazard property, which links two hazard rates independently on the baseline.
When the covariate acts equally on $X$ and $Y$, i.e. $\Phi (z)=\Psi (z)$ for all $z$, then $B^{z}(s)=A(s)$ for all values of $s$ as seen in Proposition \ref{Proposition B fonction de A}. Thus, the copula of $\overline{H}^{z}$ equals that of the baseline $\overline{H}^{\scriptscriptstyle{0}}$; the dependency structure of $X$ and $Y$ should not be altered through (\textbf{M}). Only the marginal distributions of $X$ and $Y$ in this case reflect the role of the covariate.\\
\newline
We propose some illustration. We use $\Phi (z)=e^{\alpha z}$, with $\alpha=1.5$ and $\Psi(z)=e^{\beta z}$, with $\beta=2$. Figure \ref{evolution Clayton sym} illustrates formula \eqref{prop archimedean}. We represent the change of the density of $C_{\overline{H}^z}$ with $z$. The archimedean copula is the Clayton copula whose generator is defined by $\varphi(t)=t^{-\theta}-1$. We take $\theta=3$. In this figure the model tends rapidely to independence since the density of the copula tends to $1$ as $z$ increases. Figure \ref{transition_AB} illustrates the transition formula \eqref{B dependence}. The baseline copula is the Gumbel copula with $\theta =3$. The Pickands function of the Gumbel copula is $A(t)=\left[t^\theta + (1-t)^\theta\right]^{\frac{1}{\theta}}$. The dependence functions are ordered wrt $z$. As $z$ increases, the model tends to independent marginals.
\begin{figure}[!t]
\centering
\includegraphics[scale=0.44]{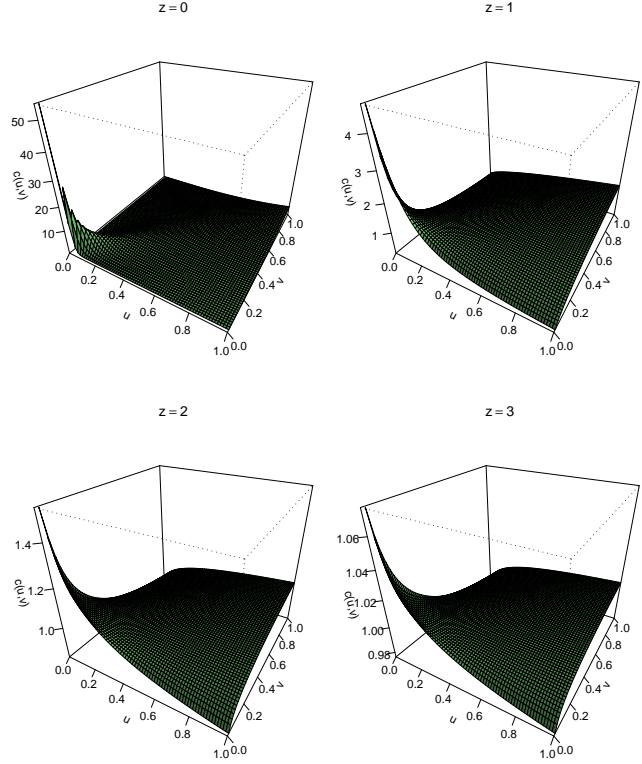}
\caption{Illustration of \eqref{asymmetric copula} with the Clayton copula density for the baseline ($z=0$)}
\label{evolution Clayton sym}
\end{figure}
\newpage
\begin{figure}[!t]
\centering
\includegraphics[height=10cm,width=10cm]{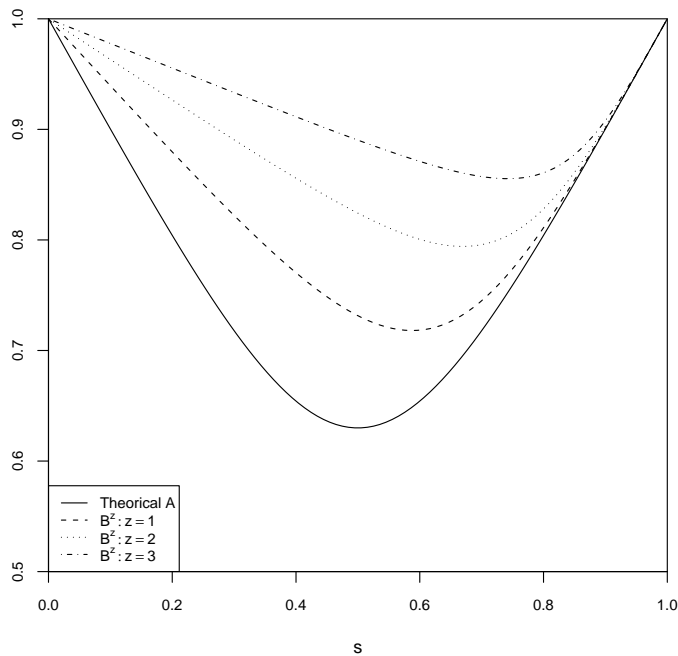}
\caption{Illustration of formula \eqref{B dependence}}
\label{transition_AB}
\end{figure}
\section{Asymmetric logistic models of dependence} \label{logistic}
This section deals with specific parametric models for dependence which are stable under (\textbf{M}). We consider model (\textbf{M}) specialized in the case when the copula of $\overline{H}^{\scriptscriptstyle{0}}$ is a Gumbel copula.\ The margins of $\overline{H}^{\scriptscriptstyle{0}}$ can be any. It is a simple parametrized model of copulas, which is an evc on one hand, and which models frailty bivariate dependence, being hence an archimedean copula. Indeed it is the only copula satisfying jointly these two properties (see \citet{Nelsen2006}, Theorem 4.5.2; \citet{Genest1989}, statement A).

The Gumbel copula writes
\begin{equation*}
C(u,v)\text{:=exp}\left[ -\left\{ \left( -\log u\right) ^{\theta }+\left(-\log v\right) ^{\theta }\right\} ^{1/\theta }\right]
\end{equation*}
with $\theta \geq 1.$ The dependence function of this copula is 
\begin{equation} \label{A_Gumbel}
A(s)=\left[ s^{\theta }+\left( 1-s\right) ^{\theta }\right] ^{1/\theta }.
\end{equation}
Assume that $\overline{H}^{\scriptscriptstyle{0}}$ has an evc with dependence function $A.$ When the covariate $z$ acts, the dependence function $B^{z}$ defined through Proposition \ref{Proposition B fonction de A} determines the \textit{asymmetric logistic }copula. This copula has three parameters $\alpha (z),\beta (z)$ and $\theta$.
Recall from (\ref{not1}), (\ref{not2}) and proposition \ref{Proposition B fonction de A} that 
\begin{equation*}
\alpha (z)=\min \left( \frac{\Psi (z)}{\Phi (z)},1\right)
\end{equation*}
and 
\begin{equation*}
\beta (z)=\min \left( \frac{\Phi (z)}{\Psi (z)},1\right).
\end{equation*}
It holds
\begin{equation} \label{Gumbel_dep}
B^{z}(s)=1-\alpha (z)+\left[ \alpha (z)-\beta (z)\right]s +\left[ \alpha(z)^{\theta }\left( 1-s\right) ^{\theta }+\beta (z)^{\theta }s^{\theta}\right] ^{1/\theta }.
\end{equation}
When $\alpha (z)=\beta (z)$ (which implies that they equal $1$), i.e. when $z$ acts equally on $X$ and $Y,$ $B^{z}(s)=A(s)$ for all $s$.
The copula of $\overline{H}^{z}$ is
\begin{equation} \label{asymmetric logistic}
C_{\overline{H}^{z}}(u,v)=\Pi (u^{1-\alpha (z)},v^{1-\beta (z)})C(u^{\alpha
(z)},v^{\beta (z)}).
\end{equation}
where $\Pi (u,v):=uv$ is the \textit{product copula} \citep[see][p.11]{Nelsen2006}.
The dependence function $B^{z}$ is an asymmetric form of the Gumbel dependence function $A$ defined in (\ref{A_Gumbel}). This is the \textit{asymmetric logistic model }in \citet{Tawn1988} when the margins are standard exponential. As developped by Khoudraji \citep[see][chap~4]{Khoudraji95} and Genest et al. in \citet{Discussion_Genest}, Proposition 3, given two dependence functions $A_{1}$ and $A_{2}$ , two constants $\kappa $ and $\eta $ with $0<\kappa ,\eta <1$ the function defined through
\begin{equation*}
B(s):=\left( \kappa s+\eta \bar{s}\right) A_{1}\left( \frac{\kappa s}{\kappa s+\eta \bar{s}}\right) +\left( \bar{\kappa }s+\bar{\eta }\bar{s}\right) A_{2}\left( \frac{\bar{\kappa }s}{\bar{\kappa }s+\bar{\eta}\bar{s}}\right)
\end{equation*}
where $\bar{s}$ denotes $1-s$, is the dependence function of the extreme value copula defined by $C_{A_{1}}(u^{1-\kappa}, v^{1-\eta })C_{A_{2}}(u^{\kappa },v^{\eta })$. Genest et al. define this procedure as a technique to generate asymmetric copulas. The class of copulas defined in (\ref{asymmetric logistic}) has been introduced by Genest et al (see their Proposition 2 in \citet{Discussion_Genest}).\\
\newline We now analyze this class of copulas in terms of frailty models. The Gumbel copula is associated with a frailty model of order 1, namely
\begin{equation*}
C(u,v)=\Lambda ^{-1}\left( \Lambda (u)+\Lambda (v)\right)
\end{equation*}
where $\Lambda^{-1}(s)=\int_{0}^{\infty }e^{-sw}dM_{1/\theta }(w) = e^{-s^{1/\theta}},\, \theta >1$, is the Laplace transform of the positive stable law $M_{1/\theta }$ on $\mathbb{R}^{+}$ with tail heaviness index $1/\theta $, location parameter $0$, scale parameter $1$ and skewness parameter $0$ (see \citet{Ravishanker} and the  example 5 in \citet{Oakes1989}). Denote $W$ a positive random variable with cdf $M_{1/\theta }.$ A bivariate sdf $\overline{H}^{\scriptscriptstyle{0}}$ with Gumbel copula $C$ writes
\begin{equation*}
\overline{H}^{\scriptscriptstyle{0}}(x,y)=\int_{0}^{\infty}\left\{\overline{F}(x)\overline{G}(y)\right\}^{w}dM_{1/\theta }(w),
\end{equation*}
for some sdf $\overline{F}$ and $\overline{G}$. Therefore $\overline{H}^{\scriptscriptstyle{0}}$ is a frailty bivariate sdf, with stable frailty measure and margins $\int_{0}^{\infty}{\overline{F}(x)}^{w}dM_{1/\theta }(w)$ and $\int_{0}^{\infty}{\overline{G}(y)}^{w}dM_{1/\theta }(w)$.\\
Let $U_{1}$ and $U_{2}$ be two independent r.v's, both independent of $W$. We assume that $U_{1}$ and $U_{2}$ have a positive stable law $M^{1}$ and $M^{2}$ with tail heaviness index $1/\theta$. The r.v $U_{1}$ (resp. $U_{2}$) has shape parameter $\frac{1}{\alpha (z)}-1$ (resp. $\frac{1}{\beta (z)}-1$). Both have location and skewness parameters $0$.
Define $S_{i}:=U_{i}+W,i=1,2$ \citep[see][]{Drouet}. Denote $\varphi _{i}^{-1}(s)$ the Laplace transforms of the distribution of $S_{i}$. Denote further $\psi _{i}^{-1}(s)$ the Laplace transform of the distribution of $U_{i}$. Let $M$ denote the probability measure of $(S_{1}$ $,S_{2})$. For arbitrary sdf $\overline{H}_1$ and $\overline{H}_2$ define the bivariate sdf
\begin{equation*}
\overline{H}(x,y):=\int \int \left\{\overline{H}_1(x)\right\}^{s_{1}}\left\{\overline{H}_2(y)\right\}^{s_{2}}dM(s_{1},s_{2})
\end{equation*}
which we call a frailty model of order 2 since it implies a bivariate latent variable. Frailty models of order two have been considered in \citet{Marshall1988} (see their formula (2.2)). The marginals of $\overline{H}$ are $\overline{F}(x) = \int_{0}^{\infty} \left\{\overline{H}_1(x)\right\}^{s_{1}}dM_{S_{1}}(s_{1})$ and $\overline{G}(y) = \int_{0}^{\infty} \left\{\overline{H}_2(y)\right\}^{s_{2}}dM_{S_{2}}(s_{2})$. \\
We prove that the copula of $\overline{H}$ is (\ref{asymmetric logistic}). Indeed
\begin{equation*}
\overline{H}(x,y)=\int \left\{\overline{H}_1(x)\right\}^{u_{1}}dM^{1}(u_{1})\int \left\{\overline{H}_2(y)\right\} ^{u_{2}}dM^{2}(u_{2})\int \left\{\overline{H}_1(x)\overline{H}_2(y)\right\}^{w}dM_{1/\theta}(w).
\end{equation*}
Introducing the Laplace transforms defined above and rewriting the marginals $\overline{F}(x)=\varphi _{1}^{-1}\left( -\log \overline{H}_1(x)\right)$ and $\overline{G}(y)=\varphi _{2}^{-1}\left( -\log \overline{H}_2(y)\right)$ we obtain the following expression for the copula of $\overline{H}$ 
\begin{equation} \label{copule H}
C(u,v)=\psi _{1}^{-1}\left( \varphi _{1}(u)\right) \psi _{2}^{-1}\left(
\varphi _{2}(v)\right) \Lambda ^{-1}\left( \varphi _{1}(u)+\varphi
_{2}(v)\right),
\end{equation}
since $\overline{H}(x,y) = C(\overline{F}(x),\overline{G}(y))$. Substituting $\psi _{i}$ and $\varphi _{i},i=1,2$ by their expressions in the above expression, noting that $\varphi _{i}^{-1}=$ $\psi _{i}^{-1}\Lambda^{-1},$ (\ref{copule H}) coincides with (\ref{asymmetric logistic}). We now prove that for an adequate choice of $\overline{H}_1$ and $\overline{H}_2$ the bivariate sdf $\overline{H}$ has same marginals as $\overline{H}^{z}$. Indeed let
\begin{eqnarray*}
\overline{H}_1(x) &:&=\exp \left[ -\left( -\min \left( \Phi (z),\Psi(z)\right) \log \overline{F}^{\scriptscriptstyle{0}}(x)\right) ^{\theta }\right] \\
\overline{H}_2(y) &:&=\exp \left[ -\left( -\min \left( \Phi (z),\Psi(z)\right) \log \overline{G}^{\scriptscriptstyle{0}}(y)\right) ^{\theta }\right]
\end{eqnarray*}
which yields $\overline{F}(t)=\overline{F}^{z}(t)$ and $\overline{G}(t)=\overline{G}^{z}(t)$ for all $t$. Therefore $\overline{H}$ and $\overline{H}^{z}$ coincide. We have proved
\begin{proposition}
When $\overline{H}^{\scriptscriptstyle{0}}$ is a frailty bivariate sdf with Gumbel copula, then for all $z$, $\overline{H}^{z}$ is a frailty sdf of order 2 with asymmetric logistic copula given in (\ref{asymmetric logistic}).
\end{proposition}
More generally we have
\begin{proposition}
The class of second order frailty models with asymmetric logistic copula is stable under (\textbf{M}).
\end{proposition}
\begin{proof}
Let $A$ denote the dependence function of an asymmetric logistic copula.
\begin{equation*}
A(s):=1-\kappa +\left( \kappa -\eta \right) s+\left[ \kappa ^{\theta }\left(1-s\right) ^{\theta }+\eta ^{\theta }s^{\theta }\right] ^{1/\theta }.
\end{equation*}
By (\ref{B dependence}) it holds
\begin{equation*}
B^{z}(s)=1-\kappa ^{\prime }+\left( \kappa ^{\prime }-\eta ^{\prime }\right)s+\left[ \kappa ^{\prime \theta }\left( 1-s\right) ^{\theta }+\eta ^{\prime\theta }s^{\theta }\right] ^{1/\theta }
\end{equation*}
with
\begin{eqnarray*}
\kappa ^{\prime } &=&\alpha (z)\kappa \\
\eta ^{\prime } &=&\beta (z)\eta.
\end{eqnarray*}
These new parameters are in $\left( 0,1\right) ,$ as are $\kappa $ and $\eta$. We have proved that the class of sdf with asymmetric logistic copula is stable under (\textbf{M}). Any sdf with such a copula is necessarily a frailty sdf of order 2. Indeed this follows from (\ref{copule H}) which enables identifying the frailty measure $M$ of $\overline{H}^{z}$ as the joint distribution of $\left( S_{1},S_{2}\right) $ as defined here above.
\end{proof}
\begin{remark}
It can be seen that the only sdf which are frailty of order 2 with evc are precisely the frailty models with asymmetric logistic copula $C(u,v)=\Pi(u^{1-\kappa },v^{1-\eta })C_{\theta }(u^{\kappa },v^{\eta })$ with $0<\eta,\kappa <1$, and where $C_{\theta }$ is the standard Gumbel copula with parameter $\theta \geq 1$.
\end{remark}
\section{Simulation results}
\subsection{Stability of the estimate under the model} \label{sim_stab}
For a given copula $C_{\theta}$ in a parametric family we simulated $N$ independent couples $(X_i,Y_i)$, $i = 1,...,N$, with joint distribution function $C_{\theta}$. Estimation of $\theta$ was performed using plug-in technique, leading $\hat{\theta}$. We repeated the procedure 1000 times. We compared $C_{\theta}^z$ with $C_{\hat{\theta}}^z$ for various $z$, given known functions $\Phi$ and $\Psi$. Here $\Phi(z) = e^{\alpha z}$ with $\alpha = 1.5$ and $\Psi(z) = e^{\beta z}$ with $\beta = 2$. The figures hereunder show the mean relative error $\int \int \left\vert \frac{C_{\theta}^z(u,v)-C_{\hat{\theta}}^z(u,v)}{C_{\theta}^z(u,v)} \right\vert dC_{\theta}^z(u,v)$ with respect to $z$, together with the 95\% confidence interval. This indicator is obtained on a grid of 100 points in $[0,1] \times [0,1]$. Figure \ref{Clay_sim1} pertains to the Clayton copula $C_\theta(u,v) = (u^{-\theta} + v^{-\theta} -1)^{-\frac{1}{\theta}}$ with $\theta = 3$. The estimate $\hat{\theta}$ is defined through $\hat{\theta} = \frac{2\hat{\tau}}{1-\hat{\tau}}$ where
\begin{equation} \label{tau_emp}
\hat{\tau} = \frac{\text{number of concordant pairs} - \text{number of discordant pairs}}{
\left(\begin{array}{c}
N\\
2
\end{array}
\right)}
\end{equation} 
is the empirical estimate of the Kendall's tau \citep[see][p.158]{Nelsen2006}. Figure \ref{Gumb_sim1} pertains to the Gumbel copula $C_\theta(u,v) = \exp[-\{(-\log u)^\theta+(-\log v)^\theta\}^\frac{1}{\theta}]$ with $\theta=3$. The estimate $\hat{\theta}$ is defined through $\hat{\theta} = \frac{1}{1-\hat{\tau}}$. It appears from those curves that a good estimate in the reference zone ($z=0$) propagates accordingly to other zones (indexed by $z$), without deteriorating the estimation accuracy. These facts also hold for very small values of $N$; obviously the larger $N$, the better the accuracy, for all $z$.
\begin{remark}\
\begin{enumerate}
\item These simulations are closely related to some industrial application where the environmental variable $z$ has range $[0,0.3]$. Obviously $\alpha$ (resp. $\beta$) is fitted accordingly around 1.5 (resp. 2), as in these simulations. Changing the scale of $z$ modified the value of $\alpha$ and $\beta$. This entails that no comparison can be performed with higher values of $z$ on these graphs, keeping in mind the underlying applied statistical context.
\item In these simulations the paramater $\theta$ of the copula is estimated using the baseline group. The estimation of the copula under $z$ is obtained through a transformation of the copula under $z=0$. Taking into account all the pooled data for all $z$ values introduces complex estimation procedures. 
\end{enumerate}
\end{remark}
\newpage
\begin{figure}[!h]
\begin{center}
\includegraphics[width=14cm]{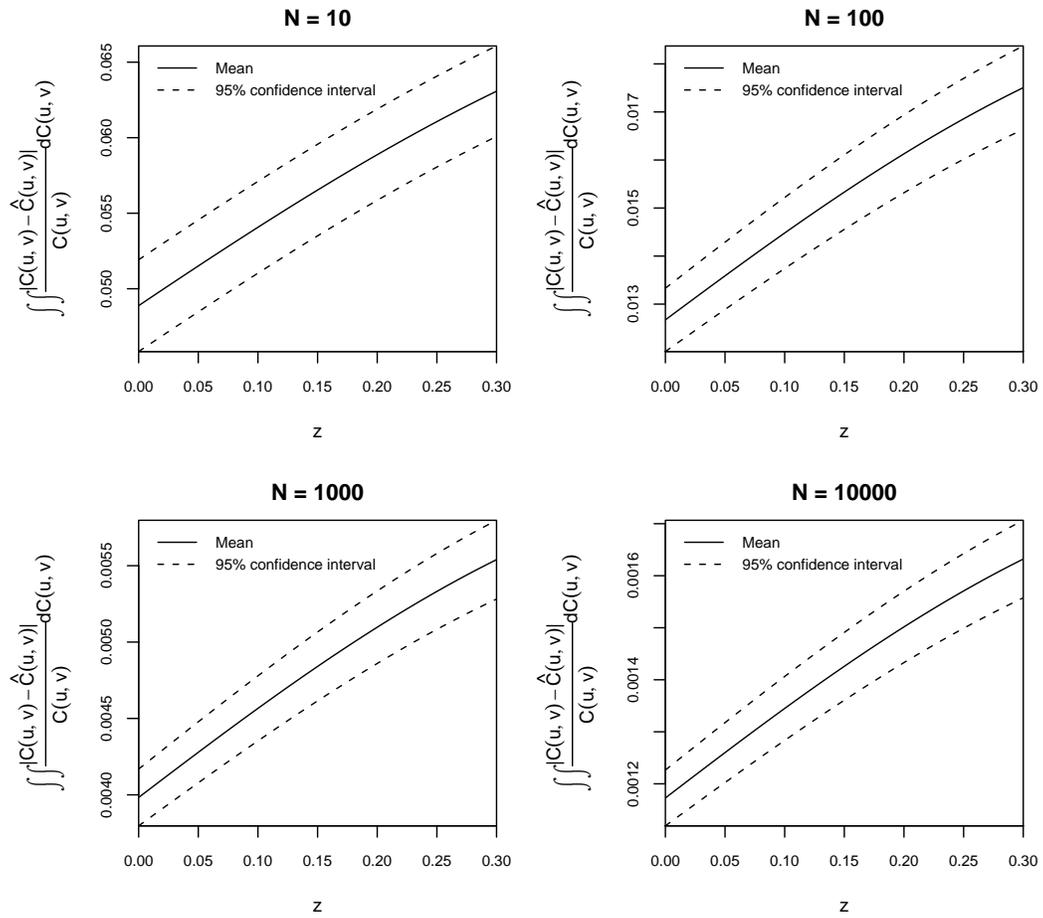}
\caption{Relative error pertaining to Clayton copula}
\label{Clay_sim1}
\end{center}
\end{figure}
\newpage
\begin{figure}[!h]
\begin{center}
\includegraphics[width=14cm]{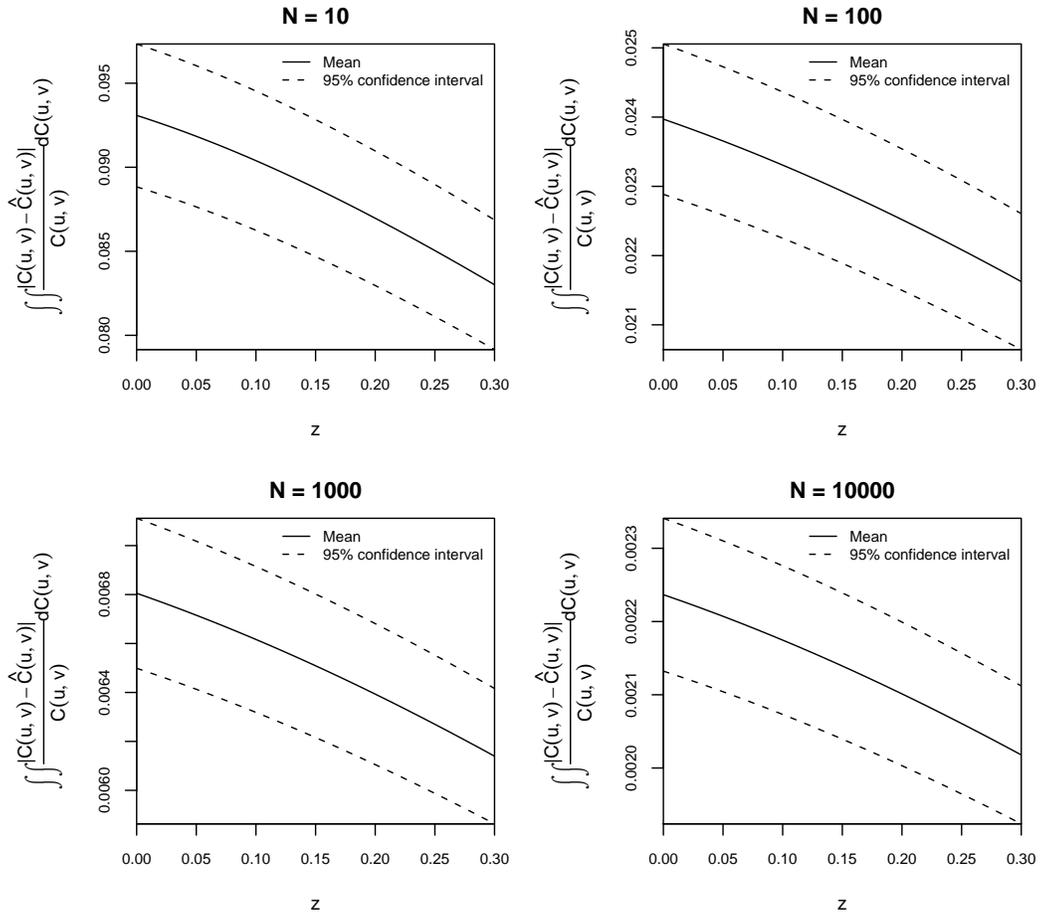}
\caption{Relative error pertaining to Gumbel copula}
\label{Gumb_sim1}
\end{center}
\end{figure}
\newpage
\subsection{A case study through simulation} \label{Cox simulation}
In this section, $z$ is bivariate, $\mathbf{z} = (z_1,z_2)$ with $z_i \in \{0;1\}$, $i=1,2$. We consider model \textbf{(M)} with Weibull marginals. In the reference zone, $\mathbf{z}^\prime=(0,0)$, the marginals are $W(2,12000)$ for $X$ and $W(1.5,8000)$ for $Y$. The copula in the reference zone is Clayton (archimedean) or Gumbel (which is both archimedean and e.v.c). The function $\Phi$ (resp $\Psi$) is $\exp(\boldsymbol{\alpha}^\prime \mathbf{z})$ (resp $\exp(\boldsymbol{\beta}^\prime \mathbf{z})$) where $\boldsymbol{\alpha}^\prime = (0.1,0.06)$ (resp $\boldsymbol{\beta}^\prime = (0.07,0.25)$). We simulated $N = 200$ couples with Gumbel copula or with Clayton copula under $\mathbf{z}^\prime=(0,0)$, $100$ values under $\mathbf{z}^\prime=(1,0)$ and $100$ values under $\mathbf{z}^\prime=(0,1)$. Simulations was performed using Khoudraji's algorithm \citep[see][]{Khoudraji95}. We estimated $\boldsymbol{\alpha}$ and $\boldsymbol{\beta}$ using the whole sample ($400$ couples) through $\text{Cox}$'s partial likelihood estimation method (implemented in R via the \textit{coxph} procedure); therefore we estimated $\boldsymbol{\alpha}$ and $\boldsymbol{\beta}$ independently by working separately with the $X_i$'s and the $Y_i$'s. The copula parameter $\theta$ under $\mathbf{z}^\prime=(0,0)$ was estimated through the plug-in of the empirical Kendall's tau. We used formula \eqref{prop archimedean} or formula \eqref{asymmetric logistic} in order to obtain an estimated copula $\hat{C}^{\mathbf{z}}$ under $\mathbf{z}^\prime \neq (0,0)$ (with $3$ parameters instead of $1$) by replacing $\alpha(\mathbf{z})$, $\beta(\mathbf{z})$ in (\ref{not1}) and (\ref{not2}) and $\theta$ by their estimators; the estimator of $\theta$ is unchanged and has been obtained under $\mathbf{z}^\prime=(0,0)$. On the other hand, these same formulas provide the theoretical copula $C^{\mathbf{z}}$ under $\mathbf{z}$. We performed the relative accuracy of the estimation scheme through $\int \int \left\vert \frac{C^{\mathbf{z}}(u,v)-\hat{C}^{\mathbf{z}}(u,v)}{C^{\mathbf{z}}(u,v)}\right\vert dC^{\mathbf{z}}(u,v)$. The procedure is repeated 1000 times. The results are given in Table \ref{sim_Cox}. We give the average of the uniform relative error over the 1000 simulations (bold character) and its 95-percent confidence interval.
\begin{table}[h!]
\begin{tabular}{|l|c|c|c|}
\hline 
 & $\mathbf{z}^\prime = (0,0)$ & $\mathbf{z}^\prime = (1,0)$ & $\mathbf{z}^\prime = (0,1)$\\ \hline
\multirow{2}{*}{Clayton ($\theta = 3$)} & $\boldsymbol{0.93\%}$ & $\boldsymbol{2.03\%}$ & $\boldsymbol{2.60\%}$ \\ 
& [0.88\%,0.97\%] & [1.95\%,2.12\%] & [2.49\%,2.71\%] \\ \hline
\multirow{2}{*} {Gumbel ($\theta = 3$)} & $\boldsymbol{0.81\%}$	& $\boldsymbol{1.55\%}$ & $\boldsymbol{2.02\%}$ \\
& [0.77\%,0.85\%] & [1.49\%,1.62\%] & [1.93\%,2.11\%] \\ \hline
\end{tabular}
\vspace {0.5cm}
\caption{Relative error for different copulas}
\label{sim_Cox}
\end{table}
\begin{table}[ht!]
\begin{tabular}{|l|c|c|c|}
\hline 
 & $\mathbf{z}^\prime = (0,0)$ & $\mathbf{z}^\prime = (1,0)$ & $\mathbf{z}^\prime = (0,1)$\\ \hline
\multirow{2}{*}{Clayton} & $\boldsymbol{3.49\%}$ & $\boldsymbol{5.73\%}$ & $\boldsymbol{11.65\%}$ \\ 
& [3.32\%,3.66\%] & [5.44\%,6.02\%] & [11.08\%,12.23\%] \\ \hline
\multirow{2}{*} {Gumbel} & $\boldsymbol{2.35\%}$	& $\boldsymbol{4.01\%}$ & $\boldsymbol{8.50\%}$ \\
& [2.24\%,2.46\%] & [3.82\%,4.22\%] & [8.11\%,8.89\%] \\ \hline
\end{tabular}
\vspace {0.5cm}
\caption{Relative error on Spearman's rho}
\label{err_spearman}
\end{table}
\newpage The estimators of $\Phi$ and $\Psi$ however are of mean accuracy, which nevertheless does not deteriorate the quality of the estimators of the copula, in all cases which we considered. In Table \ref{sim_Cox} the indicator mixes both propagation and estimation errors. The relative error w.r.t the theoretical propagated copula is very small in the case of the most commonly used baseline copulas. Measured in terms of Spearman's rho, the relative error does not deteriorate significantly either, as seen in Table \ref{err_spearman}.
\subsection{Propagation of misspecification errors}
\subsubsection{Assuming $\Phi$ and $\Psi$ known}
We simulated $N=200$ couples of r.v's with Clayton distribution function $C_\theta$ with parameter $\theta = 3$. Here $\Phi$ and $\Psi$ are as in Section \ref{sim_stab}. We estimated the Kendall's tau through the classical non parametric estimate (\ref{tau_emp}). We used a misspecified model assuming that the data have been generated under a Ali-Mikhail-Haq copula $G_{\hat{\theta}}$ \citep[][Table 4.1 p.116]{Nelsen2006} with parameter $\hat{\theta} = \frac{2}{3-\hat{\tau}}$. As seen in \cite{Nelsen2006}, both copulas present common features. Misspecification may therefore occur as presented here. For various $z$ we used formula \eqref{prop archimedean} to define both the true copula $C_{\theta}^z$ and the misspecified estimated copula $G_{\hat{\theta}}^z$. The misspecification error is defined through $Err(z) = \int \int \left\vert \frac{C_{\theta}^z(u,v)-G_{\hat{\theta}}^z(u,v)}{C_{\theta}^z(u,v)}\right\vert dC_{\theta}^z(u,v)$.
\begin{figure}[!h]
\begin{center}
\includegraphics[height=11cm,width=11cm]{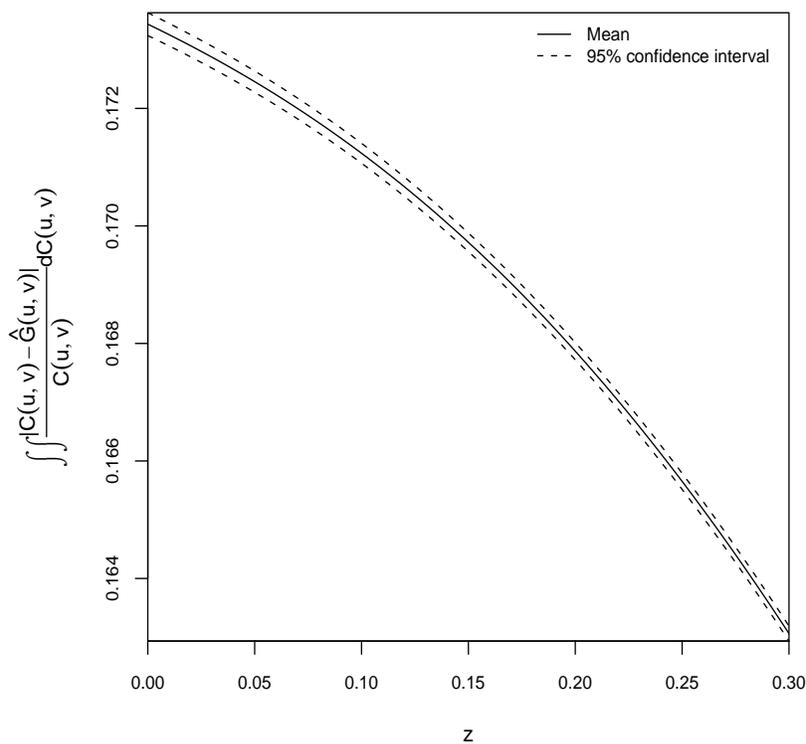}
\caption{Relative error pertaining to the misspecification when $\Phi$ and $\Psi$ are known}
\label{Clay_Gumb_sim3}
\end{center}
\end{figure}
\newpage
\subsubsection{Assuming $\Phi$ and $\Psi$ unknown}
We simulated $N=200$ couples of r.v's with Clayton distribution function $C_\theta$ with parameter $\theta = 3$. $z$ is bivariate and the functions $\Phi$ and $\Psi$ are as in Section \ref{Cox simulation}. We estimated $\Phi$ and $\Psi$ through partial likelihood and we used a misspecified model assuming that the data have been generated under a Ali-Mikhail-Haq copula $G_{\hat{\theta}}$. The results in Table \ref{sim_Cox_CG} show that the misspecification error is of order $17\%$ and keeps stable through the propagation. At the contrary Table \ref{sim_Cox} shows that the error under the true model is much smaller and propagates with great accuracy. This enlights the need for a good specification in this model.
\begin{table}[h!]
\begin{tabular}{|l|c|c|c|}
\hline
 & $\textbf{z}^\prime = (0,0)$ & $\textbf{z}^\prime = (1,0)$ & $\textbf{z}^\prime = (0,1)$\\ \hline
\multirow{2}{*} {Clayton - AMH} & $\boldsymbol{17.38\%}$	& $\boldsymbol{17.61\%}$ & $\boldsymbol{14.72\%}$ \\
& [17.36\%,17.40\%] & [17.54\%,17.67\%] & [14.65\%,14.80\%] \\ \hline
\end{tabular}
\vspace {0.5cm}
\caption{Relative error $Err(z)$ pertaining to the misspecification when $\Phi$ and $\Psi$ are unknown}
\label{sim_Cox_CG}
\end{table}
\begin{table}[h!]
\begin{tabular}{|l|c|c|c|}
\hline
 & $\textbf{z}^\prime = (0,0)$ & $\textbf{z}^\prime = (1,0)$ & $\textbf{z}^\prime = (0,1)$\\ \hline
\multirow{2}{*} {Clayton - AMH} & $\boldsymbol{53.39\%}$	& $\boldsymbol{54.85\%}$ & $\boldsymbol{55.29\%}$ \\
& [53.31\%,53.46\%] & [54.61\%,55.09\%] & [54.83\%,55.75\%] \\ \hline
\end{tabular}
\vspace {0.5cm}
\caption{Relative error on Spearman's rho}
\label{err_spear_cg}
\end{table}
\newline In Table \ref{err_spear_cg} the relative error of the Spearman's rho is presented, when a Clayton baseline copula is substituted by an AMH. The error deteriorates while propagating along $z$.
\section{Conclusion}
In this paper a new model for bivariate dependence is proposed. The margins follow a PH model; the covariate acts both on the margins and on their dependence function, handled through the copula. Such model is in contrast with other approaches in which the role of the covariate is restricted to its influence on the margins, the copula deserving a separate study. Applications of the present model fit in "father and sons" paradigm as well as in various industrial contexts. Classical families of copulas (archimedean or extreme value) are shown to be stable under the action of the covariate, in the sens that the model acts in a transitive way in those classes. In parametric classes this model provides explicit transformation of the parameter of the copula w.r.t the covariate. This model is adequate for the case of positive dependence between the margins for all the values of the covariate.
\newpage
\appendix
\section{Proof of Proposition \ref{Prop Hbar sdf}}
Trivially we show that $\lim\limits_{x_{j}\rightarrow +\infty }\overline{H}^{z}(x_{1},x_{2})=0,\,j=1,2;\;\text{and}\;\lim\limits_{\substack{ {x_{1}\rightarrow 0} \\ {x_{2}\rightarrow 0}}}\overline{H}^{z}(x_{1},x_{2})=1$
\newline
Now we will prove that $\overline{H}^{z}$ satisfies the rectangle inequality (\citet{Joe1997}, p.11) which we recall here: for all $(a_{1},a_{2}),(b_{1},b_{2})$ with $a_{1}<b_{1}$, $a_{2}<b_{2}$,
\begin{equation*}
\Delta H:=\overline{H}^{z}(a_{1},a_{2})-\overline{H}^{z}(a_{1},b_{2})-\overline{H}^{z}(b_{1},a_{2})+\overline{H}^{z}(b_{1},b_{2})\geq 0.
\end{equation*}
\begin{eqnarray*}
\Delta H=\left( \overline{F}^{\scriptscriptstyle0}(a_{1})\right) ^{\Phi(z)-\Psi (z)}\left[ \left( \overline{H}^{\scriptscriptstyle 0}(a_{1},a_{2})\right) ^{\Psi (z)}-\left( \overline{H}^{\scriptscriptstyle 0}(a_{1},b_{2})\right) ^{\Psi (z)}\right]  && \\
-\left( \overline{F}^{\scriptscriptstyle0}(b_{1})\right) ^{\Phi (z)-\Psi (z)} \left[ \left( \overline{H}^{\scriptscriptstyle0}(b_{1},a_{2})\right) ^{\Psi(z)}-\left(\overline{H}^{\scriptscriptstyle0}(b_{1},b_{2})\right) ^{\Psi
(z)}\right]  &&
\end{eqnarray*}
We denote $r(t)=\left( \overline{H}^{\scriptscriptstyle0}(t,b_{2})\right)^{\Psi (z)}-\left( \overline{H}^{\scriptscriptstyle0}(t,a_{2})\right) ^{\Psi(z)}$. If $\left( \overline{H}^{\scriptscriptstyle0}\right) ^{\Psi (z)}$ is a sdf (which implies that $\left(\overline{H}^{\scriptscriptstyle0}\right)^{\Psi (z)}$ is a 2-increasing function), then by Lemma 2.1.3 in \citet{Nelsen2006}, the function $r$ is nondecreasing. Therefore
\begin{eqnarray*}
\Delta H &=& \left(\overline{F}^{\scriptscriptstyle0}(b_{1})\right)^{\Phi(z)-\Psi(z)}r(b_{1}) - \left(\overline{F}^{\scriptscriptstyle0}(a_{1})\right)^{\Phi(z)-\Psi (z)}r(a_{1}) \\
&=& r(b_{1})\left[ \left( \overline{F}^{\scriptscriptstyle0}(b_{1})\right)^{\Phi (z)-\Psi (z)}-\frac{r(a_{1})}{r(b_{1})}\left( \overline{F}^{\scriptscriptstyle0}(a_{1})\right) ^{\Phi (z)-\Psi (z)}\right].
\end{eqnarray*}
Under (\ref{H}) and if $\Phi (z)\geq \Psi (z)>0$, both $\left( \overline{H}^{\scriptscriptstyle0}\right) ^{\Psi (z)}$ and $\left( \overline{F}^{\scriptscriptstyle0}\right) ^{\Phi (z)-\Psi (z)}$ are sdf's.\ Use the fact that  $r(b_{1})$ is negative (which holds since $\left( \overline{H}^{\scriptscriptstyle0}\right) ^{\Psi (z)}$ is a decreasing function of its
second argument) to obtain 
\begin{equation*}
r(b_{1})\left[ \left( \overline{F}^{\scriptscriptstyle0}(b_{1})\right)^{\Phi (z)-\Psi (z)}-\frac{r(a_{1})}{r(b_{1})}\left( \overline{F}^{\scriptscriptstyle0}(a_{1})\right) ^{\Phi (z)-\Psi (z)}\right] \geq 0.
\end{equation*}
Note that if (\ref{H}) does not hold, $\overline{H}^{z}$ is  still  a sdf when $\Phi (z)\geq \Psi (z) \geq 1$.
\qedsymbol
\section{Proof of Proposition \ref{Proposition archimedean copula} and \ref{Generalisation archimedean copula}}
\subsection{Proof of Proposition \ref{Proposition archimedean copula}}
Write $C_{\overline{H}^{\scriptscriptstyle 0}}(u,v)=\varphi_{\scriptscriptstyle 0}^{-1}(\varphi_{\scriptscriptstyle 0}(u)+\varphi_{\scriptscriptstyle 0}(v))$. Using (\ref{CopHz w.r.t. CopF0}) some calculus yields \eqref{prop archimedean}. We prove that $\varphi_z(t)=\varphi_{\scriptscriptstyle 0} \left(t^{\frac{1}{\min(\Phi(z),\Psi(z))}}\right)$ is also a generator for all values of $t$ in $[0,1]$.\\
\begin{enumerate}
\item $\varphi_z(1) = \varphi_{\scriptscriptstyle 0}(1) = 0$
\item $\varphi_z(t)$ is strictly decreasing in $t$, since $t^{\frac{1}{\min(\Phi(z),\Psi(z))}}$ is increasing in $t$ and $\varphi_{\scriptscriptstyle 0}(t)$ is strictly decreasing.
\item It holds $\varphi_{\scriptscriptstyle 0}^{\prime}(t) + t\varphi_{\scriptscriptstyle 0}^{\prime\prime}(t) \geq 0$ since since $C_{\overline{H}^{\scriptscriptstyle 0}}$ is $TP_2$. Let us prove that $\varphi_z^{\prime \prime}(t) \geq 0$. We have, 
\begin{eqnarray*}
\varphi_z^{\prime\prime}(t) &=& m(m-1)t^{m-2}\varphi_{\scriptscriptstyle 0}^{\prime}(t^m) + (mt^{m-1})^2\varphi_{\scriptscriptstyle 0}^{\prime\prime}(t^m) \\
&\geq& m^2t^{m-2}\left[\varphi_{\scriptscriptstyle 0}^{\prime}(t^m) + t^m \varphi_{\scriptscriptstyle 0}^{\prime\prime}(t^m)\right] \geq 0
\end{eqnarray*}
where $m=\frac{1}{\min(\Phi(z),\Psi(z))}$, which proves that $\varphi_z$ is convex.
\end{enumerate}
We now prove that $C_{\overline{H}^z}$ in (\ref{prop archimedean}) is $TP_2$. It is readily checked that the product of two $TP_2$ functions is $TP_2$.
Now $(u,v)\mapsto u^a v^b$ is $TP_2$ for all $0 \leq a,b \leq 1$ and $(u,v)\mapsto C_{\varphi_{\scriptscriptstyle{z}}}(u^a,v^b)$ is $TP_2$ for all $0 \leq a,b \leq 1$, since $\varphi_{\scriptscriptstyle{z}}$ satisfies $\varphi_{\scriptscriptstyle z}^{\prime}(t) + t\varphi_{\scriptscriptstyle z}^{\prime\prime}(t) \geq 0$. We conclude that $C_{\overline{H}^{\scriptscriptstyle{z}}}$ is $TP_2$ as a product of two $TP_2$ functions.
\qedsymbol
\subsection{Proof of Proposition \ref{Generalisation archimedean copula}}
$C_{\overline{H}^z}$ is (\ref{asymmetric copula}) through simple calculations. That $(u,v)\mapsto C_{\overline{H}^z}(u,v)$ is $TP_2$ is proved as in Proposition \ref{Proposition archimedean copula}. \qedsymbol
\section{Proof of Proposition \ref{Proposition B fonction de A}}
If $K(z)<1$ then
\begin{eqnarray*}
\hat{C}_{H^{z}}(u,v) &=&\exp \left[ \frac{\Phi (z)-\Psi (z)}{\Phi (z)}\ln u \right] \exp \left[ \left( \frac{\ln u}{\Phi (z)}+\frac{\ln v}{\Psi (z)} \right) A\left( \frac{\frac{\ln v}{\Psi (z)}}{\frac{\ln u}{\Phi (z)}+\frac{\ln v}{\Psi (z)}}\right) \right] ^{\Psi (z)} \\
&=&\exp \left[ \frac{\Phi (z)-\Psi (z)}{\Phi (z)}\ln u+\left( \frac{\Psi (z)}{\Phi (z)}\ln u+\ln v\right) A\left( \frac{\ln v}{\frac{\Psi (z)}{\Phi (z)}\ln u+\ln v}\right) \right] \\
&\underset{\scriptscriptstyle(u=v^{\frac{1}{s}-1})}{=}&\exp \left[ \left(\frac{1}{s}-1\right) \frac{\Phi (z)-\Psi (z)}{\Phi (z)}\ln v+ \left[\frac{\Psi (z)}{\Phi (z)}\left( \frac{1}{s}-1\right) +1\right] A\left( \frac{1}{\left( \frac{1}{s}-1\right) \frac{\Psi (z)}{\Phi (z)}+1}\right)\ln v \right] \\
&=&\exp \left[ \frac{\ln v}{s}\left\{\left[ (1-s)\frac{\Phi (z)-\Psi (z)}{\Phi (z)}\right] +\left[ (1-s)\frac{\Psi (z)}{\Phi (z)}+s\right] A\left( \frac{s}{(1-s)\frac{\Psi (z)}{\Phi (z)}+s}\right)\right\} \right] \\
&=&\exp \left[ \ln (uv)B_{1}^{z}\left( \frac{\ln v}{\ln uv}\right) \right]
\end{eqnarray*}
with 
\begin{equation}
B_{1}^{z}(s)=(1-s)\frac{\Phi (z)-\Psi (z)}{\Phi (z)}+\left[ (1-s)\frac{\Psi(z)}{\Phi (z)}+s\right] A\left( \frac{s}{(1-s)\frac{\Psi (z)}{\Phi (z)}+s}\right)
\end{equation}
Hence
\begin{equation*}
\begin{aligned} B_1^z(s) &= (1-s)[1-K(z)] + \left[(1-s)K(z) + s\right]A\left(\frac{s}{(1-s)K(z) + s}\right)&\\ &= 1-K(z) - s[1-K(z)] + \left[(1-s)K(z) + s\right]A\left(\frac{s}{(1-s)K(z) + s}\right)&
\end{aligned}
\end{equation*}
If $K(z)>1$
\begin{eqnarray*}
\hat{C}_{H^{z}}(u,v) &=&\exp \left[ \frac{\Psi (z)-\Phi (z)}{\Psi (z)}\ln v \right] \exp \left[ \left( \frac{\ln u}{\Phi (z)}+\frac{\ln v}{\Psi (z)} \right) A\left( \frac{\frac{\ln v}{\Psi (z)}}{\frac{\ln u}{\Phi (z)}+\frac{\ln v}{\Psi (z)}}\right) \right] ^{\Phi (z)} \\
&=&\exp \left[ \frac{\Psi (z)-\Phi (z)}{\Psi (z)}\ln v+\left( \ln u+\frac{\Phi (z)}{\Psi (z)}\ln v\right) A\left( \frac{\ln v}{\frac{\Psi (z)}{\Phi (z)}\ln u+\ln v}\right) \right] \\
&\underset{\scriptscriptstyle(u=v^{\frac{1}{s}-1})}{=}&\exp \left[ \frac{\Psi (z)-\Phi (z)}{\Psi (z)}\ln v+\left[ \left( \frac{1}{s}-1\right) + \frac{\Phi (z)}{\Psi (z)}\right] A\left( \frac{1}{\left( \frac{1}{s}-1\right) \frac{\Psi (z)}{\Phi (z)}+1}\right)\ln v \right] \\
&=&\exp \left[ \frac{\ln v}{s}\left\{\left( s\frac{\Psi (z)-\Phi (z)}{\Psi (z)}\right) +\left[ (1-s)+s\frac{\Phi (z)}{\Psi (z)}\right] A\left( \frac{s}{(1-s)\frac{\Psi (z)}{\Phi (z)}+s}\right)\right\} \right] \\
&=&\exp \left[ \ln (uv)B_{2}^{z}\left( \frac{\ln v}{\ln uv}\right) \right]
\end{eqnarray*}
with
\begin{equation}
B_{2}^{z}(s)=s\frac{\Psi (z)-\Phi (z)}{\Psi (z)}+\left[ (1-s)+s\frac{\Phi (z)}{\Psi (z)}\right] A\left( \frac{s}{(1-s)\frac{\Psi (z)}{\Phi (z)}+s}\right)
\end{equation}
Hence
\begin{equation*}
\begin{aligned} B_2^z(s) &= s\left(1-\frac{1}{K(z)}\right) + \left[\frac{s}{K(z)} + 1-s\right]A\left(\frac{s}{(1-s)K(z) + s}\right)& \\
&= - \frac{s}{K(z)}[1-K(z)] + \frac{1}{K(z)}\left[(1-s)K(z) + s\right]A\left(\frac{s}{(1-s)K(z) + s}\right)& \end{aligned}
\end{equation*}
If $K(z)=1$
\begin{equation*}
\hat{C}_{H^{z}}(u,v)=\exp \left[ \ln (uv)B_{3}^{z}\left( \frac{\ln v}{\ln uv}\right) \right]
\end{equation*}
with 
\begin{equation}
B_{3}^{z}(s)=A(s)
\end{equation}
We have proved that whatever $K(z)$
\begin{equation*}
\begin{aligned} B^z(s) &= 1-\min(K(z),1) - s\min\left(\frac{1}{K(z)},1\right)[1-K(z)]&\\ & + \min\left(\frac{1}{K(z)},1\right)\left[(1-s)K(z) + s\right]A\left(\frac{s}{(1-s)K(z) + s}\right)&\\ &=1- W(z)K(z) -
sW(z)[1-K(z)] + W(z)\left[(1-s)K(z) + s\right]A\left(\frac{s}{(1-s)K(z) + s}\right)& 
\end{aligned}
\end{equation*}
We prove that $B^{z}$ is a dependence function. It holds
\begin{eqnarray*}
B^{z}(0) &=&1-W(z)K(z)+W(z)K(z)A(0) \\
&=&1,\;\text{since $A(0)=1$}
\end{eqnarray*}
and
\begin{eqnarray*}
B^{z}(1) &=&1-W(z)K(z)-W(z)[1-K(z)]+W(z)A(1) \\
&=&1-W(z)+W(z)A(1) \\
&=&1,\;\text{since $A(1)=1.$}
\end{eqnarray*}
We prove the upper and lower bounds for $B^{z}$. \\
\newline Upper bound. Using $\ A(s)\leq 1\;$\ for all $\ s$ in$[0,1]$
\begin{equation*}
B^{z}(s)\leq 1-W(z)K(z)-sW(z)[1-K(z)]+W(z)[(1-s)K(z)+s]=1.
\end{equation*}
Lower bound. Using $A(s)\geq \max (s,1-s)$
\begin{equation*}
B^{z}(s)\geq 1-W(z)K(z)-sW(z)[1-K(z)]+W(z)\max \left[ s,(1-s)K(z)\right] .
\end{equation*}
We prove that the RHS in the above display is larger than both $s$ and $1-s$. \\
Since $\max \left[ s,(1-s)K(z)\right] \geq s$,
\begin{eqnarray*}
RHS &\geq &1-W(z)K(z)-sW(z)[1-K(z)]+sW(z) \\
&=&1-(1-s)\min (K(z),1) \geq s.
\end{eqnarray*}
Since $\max \left[ s,(1-s)K(z)\right] \geq (1-s)K(z)$,
\begin{eqnarray*}
RHS &\geq &1-W(z)K(z)-sW(z)[1-K(z)]+(1-s)W(z)K(z) \\
&=&1-sW(z)\geq 1-s
\end{eqnarray*}
as sought.
It remains to prove that $B^{z}(s)$ is a convex function. Some calculus yields
\begin{equation*}
\frac{\partial ^{2}B^{z}}{\partial s^{2}}(s)=\frac{W(z)K^{2}(z)}{[(1-s)K(z)+s]^{3}}\frac{\partial ^{2}}{\partial t^{2}}A(t)\Big|_{t=\frac{s}{(1-s)K(z)+s}}\geq 0
\end{equation*}
as sought.
\qedsymbol
\section{Proof of Proposition \ref{Proposition transition B}}
Write
\begin{equation*}
B^{z^{\prime }}(s)=1-W(z^{\prime })K(z^{\prime })-sW(z^{\prime})[1-K(z^{\prime })]+W(z^{\prime })[(1-s)K(z^{\prime })+s]A\left( \frac{s}{K(z^{\prime })(1-s)+s}\right).
\end{equation*}
In the above display it holds
\begin{equation*}
A\left( \frac{s}{K(z^{\prime })(1-s)+s}\right) =A\left( \frac{\frac{s}{(1-s)\frac{K(z^{\prime})}{K(z)}+s}}{K(z)\left(1-\frac{s}{(1-s)\frac{K(z^{\prime })}{K(z)}+s}\right)+\frac{s}{(1-s)\frac{K(z^{\prime })}{K(z)}+s}}\right).
\end{equation*}
The RHS in this latter expression can be written as a function of $B^{z}\left( \frac{s}{(1-s)\frac{K(z^{\prime })}{K(z)}+s}\right) $. Some calculus yields
\begin{equation}
\begin{aligned} B^{z'}(s) &= 1 - W(z')K(z') - sW(z')[K(z)-K(z')] - [1-W(z)K(z)]\frac{W(z')}{W(z)}\left[(1-s)\frac{K(z')}{K(z)}+s\right]&\\
&+\frac{W(z')}{W(z)}\left[(1-s)\frac{K(z')}{K(z)}+s\right] B^z\left(\frac{s}{(1-s)\frac{K(z')}{K(z)}+s}\right)& \end{aligned}
\end{equation}
which is (\ref{transition Bz'_Bz}). \qedsymbol
\newpage
\bibliographystyle{chicago}
\bibliography{Biblio_Cox}
\end{document}